\documentclass{amsart}

\usepackage{amsmath, amssymb, array, latexsym, hyperref, color}

\usepackage{fullpage}

\newtheorem{theorem}{Theorem}[section]
\newtheorem{lemma}[theorem]{Lemma}
\newtheorem{proposition}[theorem]{Proposition}
\newtheorem{corollary}[theorem]{Corollary}

\theoremstyle{definition}

\newtheorem{example}[theorem]{Example}

\theoremstyle{remark}

\newtheorem{remarks}[theorem]{Remarks}

\numberwithin{equation}{section}

\newcommand{\prend}{$\hfill \Box$}
\newcommand{\ls}{\leqslant}
\newcommand{\gr}{\geqslant}

\newcommand{\vertiii}[1]{{\left\vert\kern-0.25ex\left\vert\kern-0.25ex\left\vert #1 
    \right\vert\kern-0.25ex\right\vert\kern-0.25ex\right\vert}}

\begin{document}

\title{On the tightness of Gaussian concentration for convex functions}

\author{Petros Valettas}
\address{Mathematics Department, University of Missouri, Columbia, MO, 65211.}
\email{valettasp@missouri.edu}
\thanks{Supported by the NSF grant DMS-1612936.}


\date{June 2017.}

\begin{abstract}
The concentration of measure phenomenon in Gauss' space states that
every $L$-Lipschitz map $f$ on $\mathbb R^n$ satisfies
$$ \gamma_{n} \left(\{ x : | f(x) - M_{f} | \gr t \} \right) \ls 2 e^{
  - \frac{t^2}{ 2L^2} }, \quad t>0, $$ where $\gamma_{n} $ is the
standard Gaussian measure on $\mathbb R^{n}$ and $M_{f}$ is a median
of $f$.  In this work, we provide necessary and sufficient conditions
for when this inequality can be reversed, up to universal constants, in
the case when $f$ is additionally assumed to be convex. In particular,
we show that if the variance ${\rm Var}(f)$ (with respect to
$\gamma_{n}$) satisfies $ \alpha L \ls \sqrt{ {\rm Var}(f) } $ for
some $ 0<\alpha \ls 1$, then
$$ \gamma_{n} \left(\{ x : | f(x) - M_{f} | \gr t \}\right) \gr c e^{
  -C \frac{t^2}{ L^2} } , \quad t>0 ,$$ where $c,C>0$ are constants depending only on $\alpha$.  
\end{abstract}

\maketitle



\section{Introduction}

The concentration of measure phenomenon is by now a fundamental tool
in modern probability theory with profound impacts in many research
areas. Its significance in the local theory of normed spaces was
emphasized by V. Milman in his seminal work \cite{Mil-dvo} on almost
spherical sections of high-dimensional convex bodies. Subsequently,
applications of concentration have increased remarkably in different
fields and concentration techniques have been developed in various
contexts.  The interested reader may consult the book of Ledoux
\cite{Led-book}, the comprehensive paper of Talagrand \cite{Tal-surv}
and the recent monograph \cite{BLM} by Boucheron, Lugosi and Massart,
for further background and detailed discussions on this very interesting subject.

The prototypical example of concentration is in Gauss' space $(\mathbb
R^n, \|\cdot\|_2, \gamma_n)$ stating that any $L$-Lipschitz map $f$ on
$\mathbb R^n$ satisfies
\begin{align} \label{eq:ineq-con}
\max \left \{ \gamma_n(\{x: f(x)\gr {\rm med}(f) +t\}), \gamma_n(\{x:
f(x)\ls {\rm med}(f)-t\}) \right \} \ls \frac{1}{2} e^{-t^2/2L^2},
\quad t>0.
\end{align}
Formally, the concentration inequality \eqref{eq:ineq-con} follows
from the solution to the isoperimetric problem in Gauss' space which
was proved independently by Borell \cite{Bor-iso} and Sudakov and
Tsirel'son in \cite{ST}. The latter asserts that among all Borel sets
$A$ with given measure, half-spaces have minimal Gaussian surface
area.  Equivalently, reformulating the Gaussian isoperimetry using
enlargements of sets with respect to the Euclidean ball, one has:
\begin{align} \label{eq:ineq-iso}
\gamma_n(A+tB_2^n) \gr \Phi( \Phi^{-1}[\gamma_n(A)] +t), \quad t>0,
\end{align} for all Borel sets $A \subseteq \mathbb R^n$, where $\Phi$ denotes the cumulative 
distribution function of a standard Gaussian.

The concentration inequality \eqref{eq:ineq-con} lies at the center of
many important Gaussian inequalities such as the logarithmic Sobolev
inequality \cite{Gro}, Nelson's hypercontractive principle \cite{Nel},
the Poincar\'e inequality \cite{Chen}, Ehrhard's inequality
\cite{Ehr-sym} and more.

It is known that \eqref{eq:ineq-con} is sharp for linear
functionals. In many cases it provides the correct estimate even if
the function is far from being linear, e.g. the $\ell_p^n$-norm
$z\mapsto \|z\|_p=( \sum_{i\ls n} |z_i|^p)^{1/p}$ for $1\ls p \ls 2$.
The fact that \eqref{eq:ineq-con} is sharp for norms in the large
deviation regime $t> {\rm med}(f)$ is well known, see
e.g. \cite[Proposition 2.9]{LedTal-book}. However, there are important
examples of Lipschitz maps such as $z\mapsto \max_{j\ls n} |z_j|$,
$z\mapsto \|z\|_4$ or an ellipsoidal norm $z\mapsto \|Az\|_2, \; A\in
\mathbb R^{m\times n}$, in which the classical concentration fails to
capture the right behavior in the small deviation regime $0<t<{\rm
  med}(f)$.  Frequently, the functions under consideration are
additionally convex, e.g. norms (of vectors or matrices), suprema
of linear functionals indexed by sets, among others. In view of the
above, several questions arise naturally such as:
\begin{itemize}
\item[]{\bf (Q1)} Do convexity assumptions ensure sharper
  concentration bounds?
\item[]{\bf (Q2)} Under what conditions is the classical concentration
  inequality \eqref{eq:ineq-con} tight?
\end{itemize}

For the first question, and for the deviation below the median, a
stronger, variance-sensitive, inequality is available for convex
functions, which was established recently in \cite{PVsd}: for any
convex map $f\in L_2(\gamma_n)$, one has
\begin{align} \label{eq:con-var}
\gamma_n(\{ x : f(x)\ls {\rm med}(f) - 20t \}) \ls \frac{1}{2} e^{- t^2 / {\rm Var}(f) }, \quad t>0.
\end{align}
The improvement lies in the fact that
\begin{align} \label{eq:var-lip}
{\rm Var}_{\gamma_n}(f) \ls {\rm Lip}(f)^2,
\end{align}(which follows by \eqref{eq:ineq-con}
or by the Gaussian Poincar\'e inequality \cite{Chen}). This new type
of concentration inequality \eqref{eq:con-var} exploits the convexity
properties of the Gaussian measure, as opposed to \eqref{eq:ineq-con}
which can be explained by isoperimetry. Corresponding estimates can
therefore be proved for arbitrary log-concave measures; see
\cite{PVvar}. All of these suggest that the left and right distributional
behaviors near the median should be treated separately.

In this note, we focus on the second question. To this end let us
discuss the aforementioned examples in more detail and review the
different reasons that \eqref{eq:ineq-con} and \eqref{eq:var-lip} can
fail to be tight. In particular, for the function $f(z)=\max_{i\ls n}
|z_i|$ the reason is the {\it super-concentration phenomenon},
following Chatterjee \cite{Cha}. Recall that a function $f:\mathbb
R^n\to \mathbb R$ is said to be ${\varepsilon_n}$-super-concentrated
for some $\varepsilon_n\in (0,1)$, if
\begin{align} \label{eq:def-super}
{\rm Var}_{\gamma_n}(f) \ls \varepsilon_n \mathbb E_{\gamma_n}
\|\nabla f\|_2^2.
\end{align}  With this terminology we have that $z \mapsto \max_{i\ls n}|z_i| = \|z\|_\infty$ is $\frac{C}{\log n}$-super-concentrated\footnote{Here and everywhere else $C,c,C_1,c_1,\ldots$ stand for positive universal constants whose values may change from line to line. For any two quantities $A,B$ depending on dimension, on the parameters of the problem, etc. 
We write $A\simeq B$ if there exists universal constant $C>0$ -independent of everything- such that $A\ls C B$ and $B\ls C A$.}, 
since 
\begin{align*}
{\rm Var}[\|Z\|_\infty] \simeq \frac{1}{\log n} \quad \textrm{and}
\quad \|\nabla \|Z\|_\infty \|_2 = {\rm Lip}(\|\cdot\|_\infty)=1 \quad
      {\rm a.s.},
\end{align*} where $Z$ is an $n$-dimensional standard Gaussian vector. Moreover, the deviation is known to be described 
by a two-level behavior (see e.g. \cite{Tal-new-iso} and \cite{Sch-cube})
\begin{align*}
ce^{-C\alpha_{n,\infty}(t)} \ls \mathbb P( \left| \|Z\|_\infty - \mathbb E\|Z\|_\infty \right|  > t ) \ls Ce^{-c\alpha_{n,\infty}(t) }, \quad 
\alpha_{n,\infty}(t)=\max\{t^2, t\sqrt{\log n}\}, \; t>0.
\end{align*}
In the light of \eqref{eq:def-super} we define the {\it super-concentration constant} of $f$ as follows:\footnote{Note that the 
ratio $$\frac{1}{{\bf s}(f)^2}=\frac{\mathbb E\|\nabla f\|_2^2}{ {\rm Var}(f)}= \frac{\langle -Lf,f\rangle}{\|f\|_{L_2}^2},$$ coincides
with the gaussian Rayleigh-Ritz quotient (see e.g. \cite{Ehr-iso}) of the operator $-L$ (the generator of the Ornstein-Uhlenbeck semigroup) at $f$, 
provided $\int f=0, \; f\neq 0$. 
}
\begin{align*}
{\bf s}(f) =  \sqrt{ \frac{ {\rm Var}(f)  }{\mathbb E \|\nabla f\|_2^2 }  }.
\end{align*} 
With this notation the mapping $z\mapsto \|z\|_\infty$ is super-concentrated with ${\bf s}(\|\cdot\|_\infty) \simeq 1/ \sqrt{ \log n }$.

However, the super-concentration phenomenon is not the only reason for
the sub-optimal bounds.  In the case of $z\mapsto \|z\|_4$ or
$z\mapsto \|Az\|_2, \; A\in \mathbb R^{m\times n}$, the reason is that
\begin{align*}
\mathbb E\|\nabla f(Z) \|_2^2 \ll {\rm Lip}(f)^2.
\end{align*} More precisely, we have (see e.g. \cite{Na} and \cite{PVZ})
\begin{align*}
{\rm Var}[\|Z\|_4] \simeq \frac{1}{\sqrt{n}} \simeq \mathbb E\| \nabla
\|Z\|_4 \|_2^2, \quad {\rm whereas} \quad {\rm Lip}(\|\cdot\|_4)=1,
\end{align*} and the deviation exhibits a three-level behavior:
\begin{align*}
ce^{-C\alpha_{n,4} (t) } \ls \mathbb P( \left| \|Z\|_4-\mathbb
E\|Z\|_4\right| >t )\ls Ce^{-c\alpha_{n,4}(t)}, \quad t>0,
\end{align*} where 
$\alpha_{n,4}(t) = \max \left\{ \min\{ t^2 n^{1/2}, t^{1/2}n^{3/8} \}, t^2 \right \} $.

For the ellipsoidal norm $z\mapsto Q_A(z):= \|Az\|_2, \; A\in \mathbb
R^{m\times n}$ one may check that:
\begin{align*}
{\rm Var}[Q_A(Z) ] \simeq \mathbb E\| \nabla Q_A(Z)\|_2^2 \simeq
\frac{\|A\|_{S_4}^4}{\|A\|_{\rm HS}^2}, \quad {\rm and} \quad {\rm
  Lip}(Q_A)=\|A\|_{\rm op},
\end{align*} $\|\cdot \|_{S_4}$ is the 4-Schatten norm, $\|\cdot\|_{\rm HS}$ is the Hilbert-Schmidt norm and $\|\cdot\|_{\rm op}$ 
stands for the operator norm of the linear map $A:\ell_2^n\to
\ell_2^m$. Again in this case the deviation obeys a multiple-level
behavior:
\begin{align*}
ce^{-C\alpha_A(t)} \ls \mathbb P\left( \left| \|AZ\|_2 - (\mathbb E\|AZ\|_2^2)^{1/2}
\right| > t \right) \ls C e^{-c\alpha_A(t)}, \quad t>0,
\end{align*} where 
$\alpha_A(t) = \max\left\{ \min\{ t^2\|A\|_{\rm HS}^2/ \|A\|_{S_4}^{4}
, t \|A\|_{\rm HS}/ \|A\|_{\rm op}^2 \} , t^2/ \|A\|_{\rm op}^2
\right\}$. The right-hand side estimate in the above concentration inequality
is due to Hanson and Wright \cite{HW} and is known to hold for the
more general class of sub-gaussian random vectors with independent
coordinates, see e.g. \cite{RV-hw} for a modern exposition and the
references therein.\footnote{Usually the Hanson-Wright inequality is stated for quadratic forms, hence for the map 
$z\mapsto \|Az\|_2^2$. The reason we omit the squares here is because we discuss for Lipschitz functions.}

Thus in all of the aforementioned cases, we observe that
\begin{align*}
{\rm Var}(f) \ll {\rm Lip}(f)^2.
\end{align*}
In view of the above remarks, and for the purpose of this note, we may intorduce, for any Lipschitz map $f$, the {\it over-concentration constant} 
of $f$ as follows:
\begin{align*} 
{\bf ov}(f) =\frac{ \sqrt{ {\rm Var} (f)} }{{\rm Lip}(f)}.
\end{align*} 
With this terminology the mapping $z\mapsto \|z\|_4$ is over-concentrated (but not super-concentrated) with 
${\bf ov}(\|\cdot\|_4) \simeq 1/ \sqrt[4]{n}$ and the ellipsoidal norm $z\mapsto Q_A(z)$ is over-concentrated 
when a gap occurs at the top of the spectrum of $A$. 

The main purpose of this note is to show that this parameter quantifies the tightness of the concentration
for convex Lipschitz maps. Alternatively, note that if
\eqref{eq:ineq-con} can be reversed (up to constants) then it implies
a reversal for \eqref{eq:var-lip}, and hence is a necessary condition
for the optimality of the concentration.  We show that this condition
is also sufficient. Namely we prove the following:

\begin{theorem} \label{thm:main} 
Let $0<\alpha\ls 1$ and let $f: \mathbb R^n \to \mathbb R$ be convex and $L$-Lipschitz map. Then, we have:
\begin{align} \label{eq:con-stabl}
{\rm Var}[f(Z)] \gr \alpha^2 L^2 \; \Longrightarrow \;  \mathbb P( |f(Z) -{\rm med}(f)| \gr tL) \gr c(\alpha) \exp(-C(\alpha)t^2 ), \quad  t>0,
\end{align} where $Z$ is a standard $n$-dimensional Gaussian vector and $c(\alpha), C(\alpha)>0$ depend only on $\alpha$. 
Moreover, we can have $c(\alpha) \gr c\alpha^8$ and $C(\alpha) \ls C \alpha^{-4} \log(e/\alpha)$.
\end{theorem}

The latter can be viewed as a stability type result in the following
sense: since \eqref{eq:var-lip} holds as equality for the affine maps,
and \eqref{eq:ineq-con} is sharp for them, we measure the
concentration for functions which are far from linear but now the
``distance'' is measured in terms of the ${\bf ov}(f)$.

\smallskip

In what follows we fix the notation. We use $\zeta$ for a standard Gaussian random variable, i.e. $\zeta\sim N(0,1)$ and $Z$ for a standard Gaussian
(usually $n$-dimensional) random vector, i.e. $Z\sim N({\bf 0}, I_n)$. We write $\gamma_n$ for the $n$-dimensional standard Gaussian measure and 
simply $\gamma$ for $\gamma_1$. The symbol $\mathbb E$ or $\mathbb E_{\gamma_n}$ stands for the expectation 
and ${\rm Var}$ or ${\rm Var}_{\gamma_n}$ stands for the variance. Let $\Phi(x)= \frac{1}{\sqrt{2\pi}} \int_{-\infty}^x e^{-t^2/2}\, dt$ the 
cumulative distribution function of a standard normal. 
We write ${\rm med}(\xi)$ for a median of a random variable $\xi$.

\section{Proof of the main result}

In this Section we discuss the basic tools in order to establish \eqref{eq:con-stabl} and we finally give the 
proof in Theorem \ref{thm:main-2}. The first ingredient in our approach is Ehrhard's inequality \cite{Ehr-sym}. 
Ehrhard proved his inequality in the following form: 

\begin{theorem} [Ehrhard] \label{thm:Ehr-ineq}
Let $A,B$ be convex sets in $\mathbb R^n$. Then, for any $\lambda \in (0,1)$ we have:
\begin{align*}
\Phi^{-1} [\gamma_n( (1-\lambda )A+\lambda B) ] \gr (1-\lambda )\Phi^{-1} [\gamma_n(A)] +\lambda \Phi^{-1}[\gamma_n(B)].
\end{align*}
\end{theorem}

The above inequality has been extended to all Borel sets by Borell in \cite{Bor-Ehr}. (See \cite{IV, vHan, NP} for recent developments 
and further references). However we will not need these extensions in this work.  

An important tool in Ehrhard's work is the notion of the {\it Gaussian rearrangement} that we recall now.
Let $f:\mathbb R^n \to \mathbb R$ be a measurable function. Following Ehrhard \cite{Ehr-iso} we define the 
Gaussian rearrangement of $f$ as the generalized inverse of the map $t\mapsto \Phi^{-1}\circ \gamma_n(f\ls t)$, i.e.
\begin{align*}
f^\ast(s) = \inf\{ t: s \ls \Phi^{-1} \circ \gamma_n(f\ls t) \}.
\end{align*} 

Note that $f^\ast$ is non-decreasing and transports the measure $\gamma$ to the distribution of $f(Z)$ with $Z\sim N({\bf 0},I_n)$. 
In the following lemma we collect some basic properties of $f^{\ast}$ that we will need in the sequel. We sketch the proof of some basic facts for 
reader's convenience.   

\begin{lemma}[Ehrhard] \label{lem:basic-gauss-rearr}
Let $f$ be a measurable function on $\mathbb R^n$. Then, $f^\ast$ enjoys the following properties:
\begin{itemize}
\item [\rm a.] The map $f^\ast :\mathbb R\to \mathbb R$ is convex, if $f$ is convex.
\item [\rm b.] The modulus of continuity of $f^\ast$ satisfies $\omega_{f^{\ast}}\ls \omega_f$. In particular, if $f$ is Lipschitz, then 
$f^\ast$ is Lipschitz with $\|f^\ast\|_{\rm Lip}\ls \|f\|_{\rm Lip}$.
\item [\rm c.] For all $u\in \mathbb R$ we have $\gamma(f^\ast \ls u) = \gamma_n(f\ls u)$.
\item [\rm d.] For any $1\ls p <\infty$ one has:
\begin{align*}
\int_{\mathbb R} |(f^\ast)' |^p \, d\gamma \ls \int_{\mathbb R^n} \|\nabla f\|_2^p \, d\gamma_n. 
\end{align*}
\end{itemize}
\end{lemma}

\noindent {\it Proof.} (a). Applying Ehrhard's inequality (Theorem \ref{thm:Ehr-ineq}) for $A=B=\{f\ls t\}$ we obtain 
that ($f^\ast)^{-1}$ is concave. 

\smallskip

\noindent (b). The isoperimetric inequality \eqref{eq:ineq-iso} implies:
\begin{align*}
\Phi^{-1}\circ \gamma_n(A+sB_2^n) \gr \Phi^{-1}\circ \gamma_n(A) +s, \quad s>0.
\end{align*} For $A=\{f\ls t\}$ says that $(f^\ast)^{-1}(t +\omega_f(s)) \gr (f^\ast)^{-1}(t)+s$, which yields:
\begin{align*}
f^\ast\left( (f^\ast)^{-1}(t)+s\right) - t \ls \omega_f(s),
\end{align*} for all $t\in \mathbb R$ and $s>0$. 

\smallskip

\noindent (c). We may write:
\begin{align*}
\gamma(\{ t\in \mathbb R : f^\ast(t) \ls u \}) = \gamma (\{ t: F(u) \gr t\})= \Phi(F(u)) =\gamma_n(\{ x\in \mathbb R^n : f(x) \ls u\}),
\end{align*} for all $u\in \mathbb R$. 

\smallskip

\noindent (d). For a proof the reader is referred to \cite{Ehr-iso} (see also \cite{Ehr-log-sob} for a related application). \prend

\medskip

The second ingredient in our approach is a remarkable inequality of Talagrand \cite{Tal-russo} that improves upon 
the classical Poincar\'e inequality \cite{Chen} (see also \cite{Cher}). Before stating his result we need to recall some basic definitions. 
A function $\psi: [0,\infty) \to \mathbb R$ is said to be a Young function if it is convex, increasing and $ \psi(0)=0$. 
If $(\Omega,\mathcal A, \mu)$ is a probability space and $\psi$ is a Young function, 
the Orlicz norm of an $\mathcal A$-measurable function $h$ in $L^{\psi}(\mu)$ is defined by 
$$ \| h\|_{\psi} := \inf \left \{ \lambda >0: \mathbb E_\mu \psi \left( \frac{ |h|}{ \lambda}\right) \ls 1 \right \} . $$ 

With this notation we have the following:

\begin{theorem} [Talagrand]
Let $\varphi(t) := \frac{ t^{2} }{ \log{ (e+t)} }, \;  t\gr 0$. For any smooth function $f$ we have:
\begin{align*}
{\rm Var}( f) \ls C \sum_{i=1}^n \|\partial_i f\|_{\varphi}^2.
\end{align*}
\end{theorem}

Talagrand proved and stated his theorem in the case of the discrete cube (with the normalized counting measure) instead of the Gaussian. 
The above statement follows by a standard application of the central limit theorem or by mimicking his proof for the cube in Gauss' space 
(see also \cite{CL} for alternative versions). 

On the real line Bobkov and Houdr\'e in \cite{BH} proved that for convex functions the reverse inequality is also true.

\begin{theorem} [Bobkov-Houdr\'e]
\label{Bobkov-Houdre} There exist universal constants $c_{1}, c_{2}, c_{3}>0$ such that 
for every convex function $f:\mathbb R\rightarrow \mathbb R$ one has
\begin{align*}
c_{1} \int_{\mathbb R} \frac{ | f^{\prime}(t)|^{2} }{ 1+ t^{2}} d\gamma(t)  \ls {\rm Var}_\gamma(f) \ls   c_{2}   \| f^{\prime}\|_{\varphi}^{2} \ls c_{3} \int_{\mathbb R} \frac{ | f^{\prime}(t)|^{2} }{ 1+ t^{2}} d\gamma(t)  .
\end{align*}
\end{theorem}

\noindent Actually we are going to use only the last inequality which is a consequence of Talagrand's theorem
(see \cite[Lemma 5]{BH}). We will also need the following:

\begin{lemma} \label{lem:key}
Let $g$ be a non-decreasing and convex map on $\mathbb R$. Then, we have:
\begin{enumerate}
\item For all $p>0$,
\begin{align*}
\|(g-{\rm med}(g))_+ \|_{L_p(\gamma)}^p \gr  \sigma_p^p  [g'(0+)]^p,
\end{align*} where $\sigma_p^p= \frac{ 2^{p/2} }{2\sqrt{\pi}} \Gamma(\frac{p+1}{2})=: \frac{1}{2}\mathbb E |\zeta |^{p}, \; \zeta\sim N(0,1)$.
\item Let $s :=\sqrt{{\rm Var} (g)}/\|g'\|_{L_2(\gamma)}$. Then 
 \begin{align*}
{\rm Var}[g(\zeta) ] \ls C_1 \left[ g' \left( \frac{C_1}{s} - \right) \right]^2.
\end{align*} 
\item  For any $t> 0$ we have:
\begin{align*}
\mathbb P \left( g(\zeta) -g(0)\gr t \sqrt{{\rm Var}[g(\zeta)] } \right) \gr 1-\Phi \left(C_1 \left(\frac{1}{s} +t \right) \right),
\end{align*} where  $C_1>0$ is a universal constant.
\end{enumerate}
\end{lemma}

\begin{proof} Note that convexity of $g$ shows that $g(t)-g(0)\gr t g'(0+)$ for all $t>0$. Combining this with the monotonicity we obtain: 
\begin{align*}
\|(g-{\rm med}(g) )_+\|_p^p = \int_0^\infty [g(t)-g(0)]^p \, d\gamma (t) \gr \frac{(g'(0+))^p}{\sqrt{2\pi}} \int_0^\infty t^p e^{-t^2/2}\, dt,
\end{align*} for all $p>0$. This proves (1).

\smallskip

By  Theorem \ref{Bobkov-Houdre} we have:
\begin{align*}
\frac{1}{c_{3}} {\rm Var}[g(\zeta)]  \ls \int_{\mathbb R} \frac{(g'(t))^2}{1+t^2}\, d\gamma(t) & \ls g'(\lambda -)^2 \int_{\mathbb R}\frac{d\gamma(t)}{1+t^2} +\int_\lambda^\infty \frac{g'(t)^2}{1+t^2} \, d\gamma(t) \\
& \ls  (g'(\lambda -))^2 + \|g'\|_{L_2}^2 /\lambda^2,
\end{align*} for all $\lambda >0$. We choose $\lambda=\sqrt{2c_{3}} /s$. This proves (2).

\smallskip 

Let $s=\sqrt{ {\rm Var}(g)}/ \|g'\|_{L_2}$ and $a= C_1/s$, where $C_1>0$ is the constant from 
part (2). Note that for any $x >a$ we have:
\begin{align} \label{eq:cvx}
g(x) \gr g(a) + (x-a)g'(a+)  \gr g(0)+ \frac{\sqrt{{\rm Var}(g) } }{\sqrt{C_1}} (x-a),
\end{align} where we have used part (2) and the monotonicity of $g$. Hence, for any $t> 0$, we obtain:
\begin{align*}
\left\{ x: x > a+t\sqrt{C_1} \right\} \subseteq \left\{x: g(x)-g(0) \gr t\sqrt{ {\rm Var}[g(\zeta)]} \right\}
\end{align*} where we have used \eqref{eq:cvx}. Finally,
\begin{align*}
\mathbb P\left(g(\zeta)-g(0) \gr t\sqrt{ {\rm Var}[g(\zeta)] } \right) \gr 1-\Phi \left( a+\sqrt{C_1} t \right),
\end{align*} for all $t>0$. This completes the proof.
\end{proof}

\smallskip
 
For any Lipschitz map $f$ on $\mathbb R^n$, recall the {\it over-concentration} and the {\it super-concentration} constant:
\begin{align*}
{\bf ov}(f) = \frac{ \sqrt{ {\rm Var}[f(Z)]} }{ {\rm Lip}(f) } \quad {\rm and} \quad {\bf s}(f)= \frac{ \sqrt{ {\rm Var}[f(Z)] }}{ \left(\mathbb E \| \nabla f(Z)\|_2^2 \right)^{1/2}}.
\end{align*} Note that in view of the Gaussian Poincar\'e inequality \cite{Chen} we have:
\begin{align} \label{eq:comp-s-ov}
{\bf ov}(f) \ls {\bf s}(f) \ls 1.
\end{align}

Our first main result is the following inequality:

\begin{proposition} \label{prop:rev-conc-cvx}
Let $f:\mathbb R^{n} \rightarrow \mathbb R$ be a convex function with $f\in L_2(\gamma_n)$. Then, 
\begin{align*} 
\mathbb P \left(f(Z) \gr M +t\sqrt{ {\rm Var} [f(Z)]} \right) \gr 1-\Phi \left( C \left( t+\frac{1}{{\bf s}(f)}\right) \right), \quad t>0,
\end{align*}
In particular, if $f$ is not super-concentrated, i.e. ${\bf s}(f) \simeq 1$, we have
\begin{align*}
\mathbb P\left( |f(Z)-M| \gr t \sqrt{ {\rm Var} [f(Z)] }\right) \gr ce^{-Ct^2}, \quad t>0, \end{align*} 
where $M={\rm med}(f)$ and $C,c >0$ are universal constants. 
\end{proposition}

\begin{proof} Let $f^\ast$ be the Gaussian rearrangement 
of $f$. Then, $f^\ast$ is convex, non-decreasing, equi-measurable with $f$ (Lemma \ref{lem:basic-gauss-rearr}). 
It follows that:
\begin{align*}
\mathbb P \left( f(Z) -M \gr t \sqrt{ {\rm Var} (f) } \right) = \mathbb P \left( f^\ast(\zeta)-f^\ast(0) \gr t \sqrt{ {\rm Var}(f^\ast) } \right) 
\gr 1-\Phi \left(C \left(t+\frac{1}{{\bf s}(f^\ast)} \right) \right),
\end{align*} for all $t>0$, where in the last step we have used Lemma \ref{lem:key} (3). 
Again by Lemma \ref{lem:basic-gauss-rearr} we have that  ${\bf s}(f) \ls {\bf s}(f^\ast)$, and the result follows. \end{proof}

Using the above distributional inequality we may derive lower estimates for the centered moments in a standard fashion.

\begin{corollary} \label{cor:ctrd-moms}
Let $f:\mathbb R^n \rightarrow \mathbb R$ be a convex function with $f\in L_2(\gamma_n)$ and $M={\rm med}(f)$. 
Then for every $ p\gr C/ {\bf s}(f)^2$ we have 
\begin{align} \label{eq:ctrd-moms-1}     \left( \mathbb E | f(Z) - M |^p \right)^{1/p}   \gr c_{1} \sqrt{p}  \sqrt{ {\rm Var}[f(Z)] } ,
\end{align} while for $ 2\ls p \ls C/ {\bf s}(f)^2$ we get
 \begin{align} \label{eq:ctrd-moms-2}
   \left( \mathbb E | f(Z) - M |^{p} \right)^{1/p}   \gr c_2  {\bf s}(f) \sqrt{p}  \sqrt{{\rm Var}[f(Z)] } .
\end{align} Moreover, if $f$ is also Lipschitz we obtain
 \begin{align*} 
    \left( \mathbb E | f(Z) - M |^{p} \right)^{1/p} \gr c_2' {\bf ov}(f) {\bf s}(f) \sqrt{p} {\rm Lip}(f) , \quad  p\gr 2,
\end{align*}
where $C, c_1, c_2, c_2' > 0$ are universal constants.
\end{corollary}

\begin{proof}
By Proposition \ref{prop:rev-conc-cvx} we have that for every $t\gr 1/ {\bf s}(f)$,
$$\mathbb P \left( f (Z) \gr M + t \sqrt{ {\rm Var}[f(Z)]} \right) \gr \Phi\left(- C_{1}t \right) \gr e^{-C_2 t^2} .$$ 
It follows that
\begin{align*} 
\mathbb E | f(Z) - M |^{p} &= p  \left( {\rm Var} [f(Z)] \right)^{p/2} \int_0^\infty t^{p-1} \mathbb P\left( |f (Z) - M| \gr t \sqrt{ {\rm Var} [f(Z)]}\right) dt \\
&\gr p  \left( {\rm Var} [f(Z)] \right)^{p/2} \int_{1/{\bf s}(f)}^\infty t^{p-1} e^{-C_2t^2}  dt \\
&= \frac{p}{2} \left( C_2^{-1} {\rm Var} [f(Z)] \right)^{p/2} \int_{C_2/{\bf s}(f)^2}^\infty t^{\frac{p}{2}-1} e^{-t} \, dt.
\end{align*}
Using the elementary inequality 
\[ a>0 , \, q\gr a+1, \quad  q\int_a^{\infty} t^{q-1} e^{ - t} \, dt \gr (q/e)^q, \] 
we conclude the assertion for $p\gr \frac{2C_2}{ {\bf s}(f)^2 }$. For $ 2\ls p \ls \frac{2C_2}{ {\bf s}(f)^{2}}$ we clearly have 
$$  \left( \mathbb E | f(Z) - M |^{p} \right)^{1/p} \gr  \frac{ \sqrt{p} \sqrt{{\rm Var}[f(Z)] } }{ \sqrt{p} } 
\gr  \frac{{\bf s}(f) }{\sqrt{2C_2}} \sqrt{p} \sqrt{{\rm Var}[f(Z)]} .$$
If $f$ is also Lipschitz the assertion follows immediately by using the definitions. 
\end{proof}

We are ready to prove our second main result. As a consequence we obtain the announced Theorem \ref{thm:main} in the Introduction. 
In fact we prove something slightly more:

\begin{theorem} \label{thm:main-2}
Let $f:\mathbb R^n\to \mathbb R$ be a convex, Lipschitz map with $L={\rm Lip}(f)$. Then, we have the following: 
\begin{align} \label{eq:rev-conc-s}
\mathbb P\left( |f(Z) - {\rm med}(f)| \gr t L \right) \gr 
c \tau^4 e^{-C \frac{t^2}{ \tau^2} \log (e/\tau ) }, \quad t>0,
\end{align} where $\tau={\bf ov}(f) {\bf s}(f)$. In particular, we get:
\begin{align} \label{eq:rev-conc-lip}
\mathbb P \left( | f(Z) -{\rm med}(f) | \gr t L \right) \gr c ({\bf ov}(f))^8 e^{- C \frac{t^2}{({\bf ov}(f) )^4} \log(e/ {\rm ov}(f))}, \quad t>0,
\end{align} where $C,c>0$ are universal constants.
\end{theorem}

\begin{proof} We prove only the first assertion, since the particular case follows from \eqref{eq:rev-conc-s} by taking into account \eqref{eq:comp-s-ov}. For this end, set $M={\rm med}(f)$ and note that from Corollary \ref{cor:ctrd-moms} we have:
\begin{align} \label{eq:rev-2}
\left( \mathbb E | f(Z) - M |^{p} \right)^{1/p}   \gr c_1  \tau \sqrt{p} L, \quad p \gr 2.
\end{align}
Set $t_0:= c_1 \tau / \sqrt{2}$. Then, for any $t>0$ we may choose $p=p(t)= 2\max\{ 1, (t/t_0)^2\}$ and take into account \eqref{eq:rev-2} to write:
\begin{align*}
\mathbb P \left( |f(Z)- M| \gr tL\right) \gr \mathbb P \left( |f(Z)-M| \gr \frac{1}{2} \|f-M\|_p \right) \gr (1-2^{-p})^2 \left( \frac{\|f-M\|_p}{\|f-M\|_{2p}} \right)^{2p},
\end{align*} where in the last step we have used the Paley-Zygmund inequality, see e.g. \cite{BLM}. 
One more application of \eqref{eq:rev-2} in conjunction with $\|f-M\|_{2p}\ls C_1 \sqrt{p} L$, yields:
\begin{align*}
\mathbb P \left( |f(Z)- M| \gr tL\right) \gr \frac{1}{2} (c_2\tau)^{2p} \gr c_3 \tau^4 \exp(- C_3 (t/t_0)^2 \log(e/\tau) ),
\end{align*} as required. The proof is complete.
\end{proof}

Summarizing we conclude the following characterization of concentration in terms of the Lipschitz constant.

\begin{corollary} \label{cor:LMS-cvx}
Let $f:\mathbb R^n \to \mathbb R$ be convex and Lipschitz map with $L={\rm Lip}(f)$. The following are equivalent: 
\begin{itemize}
\item [a.] For every $t>0$ we have:
\begin{align*}
\mathbb P( |f(Z) - {\rm med}(f) |\gr t L) \gr a_1\exp(- t^2/ A_1^2 ).
\end{align*}
\item [b.] For all $p\gr 2$ we have:
\begin{align*}
\left( \mathbb E|f(Z)- {\rm med}(f) |^p \right)^{1/p} \gr A_2 \sqrt{p} L.
\end{align*}
\item [c.] We have:
\begin{align*}
{\rm Var}[f(Z)] \gr A_{3}^{2} L^2,
\end{align*}
\end{itemize} where $Z\sim N({\bf 0},I_n)$ and the constants $a_1,A_1, A_2, A_3>0$ depend on each other.
\end{corollary}

\begin{proof}
Note that the implications (a) $\Rightarrow$ (b) $\Rightarrow$ (c) are immediate and they hold for any measurable 
function.  The implication (c) $\Rightarrow$ (a) follows from Theorem \ref{thm:main-2}. \end{proof}

\begin{remarks} 1. All previous results can be equivalently stated with the mean, in the light of 
\begin{align*}
\frac{1}{2} \left( \mathbb E |\xi -\xi' |^p \right)^{1/p} \ls \left( \mathbb E|\xi -m|^p \right)^{1/p} \ls 
2 \left( \mathbb E| \xi- \mathbb E\xi |^p \right)^{1/p} \ls 2 \left( \mathbb E|\xi -\xi'|^p \right)^{1/p} ,
\end{align*} for all $1\ls p<\infty$, where $\xi$ is any random variable, $m$ a median of $\xi$, and $\xi'$ an independent copy
of $\xi$. 

\smallskip

\noindent 2. It might worth mentioning that \eqref{eq:ctrd-moms-1}, 
should be compared with the known fact for norms \cite{LMS}:
\begin{align*}
\left( \mathbb E \left| h(Z)- M\right|^p\right)^{1/p} \gr c\sqrt{p} {\rm Lip}( h ), \quad p\gr k(h),
\end{align*} for any norm $h$ on $\mathbb R^n$, where $k(h)= ( \mathbb E[ h(Z)] / {\rm Lip}(h ) )^2$.

\smallskip

\noindent 3. In the range $t>1$, we obtain dependence $A_1\gr cA_3$ which is clearly optimal. Furthermore, if ${\bf s}(f)\simeq 1$, the above
dependence also holds for the full range of $t$. However, this may suggest that 
the restriction on $t$ in Proposition \ref{prop:rev-conc-cvx} (or the restriction on $p$ in Corollary \ref{cor:ctrd-moms}) is redundant. 
The following example shows that this is not the case. 
\end{remarks}

\begin{example}
Let $\alpha \gg 1$ and let $g_\alpha: \mathbb R \to \mathbb R$ be the function defined by:
\begin{align*}
g_\alpha(t)= c_\alpha(t-\alpha)_+, \quad c_\alpha=(1-\Phi(\alpha))^{-1/2}.
\end{align*} Then, ${\bf s}(g_\alpha)\simeq \alpha^{-1}$ and 
\begin{align*}
\mathbb P \left( g_\alpha -{\rm med}(g_\alpha) > t\sqrt{ {\rm Var} (g_\alpha)} \right) \ls 1-\Phi \left(\alpha + \frac{ct}{\alpha c_\alpha} \right),
\end{align*} for all $t>0$.

Indeed; we have the asymptotic estimate
\begin{align}
\int_\alpha^\infty \frac{d\gamma(t)}{1+t^2} \sim \frac{1-\Phi(\alpha)}{\alpha^2}, \quad \alpha\to \infty.
\end{align} Thus, we may write:
\begin{align*}
\int \frac{(g_\alpha' (t) )^2}{1+t^2} \, d\gamma(t) = c_\alpha^2 \int_\alpha^\infty \frac{d\gamma(t)}{1+t^2} \simeq \alpha^{-2}, \quad \alpha\gr 2.
\end{align*} We may compute that:
\begin{align*}
{\rm Var}(g_\alpha) \simeq \int \frac{(g_\alpha'(t))^2}{1+t^2} \, d\gamma(t) \simeq \alpha^{-2}.
\end{align*} In addition we have:
\begin{align*}
\int (g_\alpha'(t))^2 \, d\gamma(t) = c_\alpha^2 \int_\alpha^\infty d\gamma(t)=1.
\end{align*} It follows that $s(g_\alpha) \simeq 1/ \alpha$ whereas $g_\alpha'(t)=0$ for $t<\alpha$ and 
$g_\alpha'(t)=c_\alpha >0$ for $t>\alpha$. \prend

\end{example}

In particular, the above example shows that for $t \simeq \frac{1}{{\bf s}(g_\alpha)}$ the estimate in Proposition \ref{prop:rev-conc-cvx} 
is attained and that one cannot expect super-gaussian behavior for $t\ll \frac{1}{ {\bf s}(g_\alpha)}$. The discussion shows that there are mainly 
two reasons for which the classical concentration may fail to give the correct asymptotics. First the super-concentration constant
may affect the range of $t$'s and second the over-concentration constant which is apparent on the lower estimate.

\subsection{More comments on the method}

In this subsection we discuss further applications of the methods and techniques used in our result. Mainly, we present
several applications of Ehrhard's inequality which we believe are of independent interest.

\subsubsection{On the skewness of Gaussian distribution for convex functions}
Here we show an immediate consequence of Ehrhard's inequality in the spirit of Kwapien's remark from \cite{Kwa}. The author in 
\cite{Kwa} shows that for any convex function $f$ on $\mathbb R^n$ the expectation of $f(Z), \; Z\sim N({\bf 0},I_n)$ is at least as large as
its median. This fact can be interpreted as the distribution of $f(Z)$ being right-skewed. 
Another fact which illustrates this behavior is that the distribution must deviate less below its median than above its median, 
which is intuitively clear. Next statement is a rigorous proof of this fact.

\begin{proposition}
Let $f:\mathbb R^n\to \mathbb R$ be a convex map. Then, for any $t>0$ we have:
\begin{align*}
\mathbb P(f(Z) \ls {\rm med}(f)-t) \ls \mathbb P(f(Z) > {\rm med}(f)+t), \quad Z\sim N({\bf 0}, I_n).
\end{align*}
\end{proposition}

\noindent {\it Proof.} Let $m={\rm med}(f)$. We may assume without loss of generality that $\mathbb P(f(Z)\ls m)=1/2$ (otherwise $\inf f=m$
and there is nothing to prove). The map $t\mapsto g(t)= \Phi^{-1}\circ \gamma_n( f\ls t)$ is concave and $g(m)=0$. Thus,
\begin{align*}
\frac{g(m-t)+g(m+t)}{2} \ls g\left(\frac{(m-t) +(m+t)}{2} \right) =g(m)=0.
\end{align*} Finally, recall the property 
$$-g(s) = - \Phi^{-1} [\gamma_n( f\ls s)] = \Phi^{-1}[1-\gamma_n(f\ls s) ] = \Phi^{-1} [\gamma_n(f> s)], \quad s\in \mathbb R.$$ 
Therefore, 
\begin{align*}
\Phi^{-1} \circ \mathbb P( f(Z) \ls m-t) =g(m-t) \ls -g(m+t) = \Phi^{-1} \circ \mathbb P( f(Z)  > m+t).
\end{align*} The result follows by the monotonicity of $\Phi^{-1}$. \prend

\subsubsection {A small deviation inequality revisited.} 

The following theorem has been proved 
in \cite{PVsd} with a worst (universal) constant. The proof in \cite{PVsd} uses 
again Ehrhard's inequality but one works with the inverse of the Gaussian rearrangement. Here we give an alternative short 
proof using Gaussian rearrangements directly and we obtain the optimal constant.
\begin{theorem}
Let $f:\mathbb R^n\to \mathbb R$ be a convex map with $f\in L_1(\gamma_n)$. Then, we have
\begin{align*}
\gamma_n\left( \left\{ f- {\rm med}(f) < -t \| (f- {\rm med}(f))_+\|_{L_1(\gamma_n)} \right \} \right) \ls \Phi \left( -\frac{ t}{\sqrt{2\pi}} \right) , \quad t>0. 
\end{align*} The equality is attained for affine maps $x\mapsto \langle x, u \rangle + v$, where $u,v \in \mathbb R^n, \; u\neq 0$.
\end{theorem}

\noindent {\it Proof.} We may assume without loss of generality that $\|(f-{\rm med}( f))_+\|_1>0$. 
We introduce the function $f^\ast$. Note that $f^\ast(0)={\rm med}(f^\ast) ={\rm med}(f) =m$ and 
$\| (f^\ast -m)_+ \|_{L_1(\gamma)}= \| (f-m)_+\|_{L_1(\gamma_n)}$. Thus, for $u>0$ we may write:
\begin{align*}
\gamma_n(x \in \mathbb R^n : f(x)-m \ls -u) =\gamma(s \in \mathbb R : f^\ast(s) -f^\ast(0) \ls -u) 
\ls \gamma (s : (f^\ast)'(0-) \cdot s \ls -u),
\end{align*} where we have used the facts that $f^\ast$ and $f$ are equi-measurable and $f^\ast(s)\gr f^\ast(0)+(f^\ast)'(0-) \cdot s$ 
for all $s <0 $, since $f^\ast$ is convex. Lemma \ref{lem:key} shows that:
\begin{align} \label{eq:L-1-der}
\|(f^\ast - m)_+\|_{L_1(\gamma)} \gr \frac{(f^\ast)'(0+)}{\sqrt{2\pi}}.
\end{align} Thus, for $u=t\sqrt{2\pi } \|f-m\|_{L_1(\gamma_n)}$ ($t>0$ fixed) we may write:
\begin{align*}
\gamma_n \left( x: f(x) -m \ls -t \sqrt{2\pi} \|(f-m)_+\|_{L_1} \right) 
\ls \gamma \left( s: (f^\ast)'(0-) \cdot s \ls -t \sqrt{2\pi} \| (f^\ast - m)_+\|_{L_1} \right)  \ls \Phi(-t),
\end{align*} where in the last step we have used \eqref{eq:L-1-der}. \prend

\smallskip

It is easy to check that the convexity assumption in the above theorem is essential. Consider the following:

\begin{example}[The convexity cannot be omitted]
Consider the sequence of functions $g_k(t)=t^{2k+1}, \; k=1,2,\ldots$ and note that ${\rm med}(g_k)=g_k(0)=0$. 
Also, $\mathbb E (g_k-{\rm med}(g_k))_+ = \mathbb E (g_k)_+= \frac{2^k k!}{\sqrt{2\pi}}$. Thus,
\begin{align*}
\mathbb P \left( g_k(\zeta) <-t \sqrt{2\pi} \mathbb E(g_k(\zeta))_+ \right) = \Phi \left( - (t2^k k!)^{\frac{1}{2k+1}} \right).
\end{align*} The latter is smaller than $\Phi(-t)$ only when $t < \sqrt{ 2 (k!)^{1/k} } \ls \sqrt{2k}$. 
\end{example}


\subsubsection {Inequalities for the $\chi^{2}$ distribution.} 
Here we present one more application of Ehrhard's inequality. We show that the property that the 
mapping $t\mapsto \Phi^{-1}\circ \gamma_n(f\ls t)$ is concave
is shared by other significant distributions at the cost of fairly restricting the class of convex functions. Namely, 
we have the following:

\begin{proposition} \label{prop:Ehr-chi}
Let $k\in \mathbb N, k\gr 2$ and let $W=(w_1, \ldots, w_n)$ be a  random vector with independent coordinates such 
that $w_j\sim \chi^2(k)$. 
For any function $f:\mathbb R_+^n\to \mathbb R$ which is coordinatewise non-decreasing and convex, the mapping
\begin{align*}
t\mapsto \Phi^{-1} \circ \mathbb P( f(W) \ls t),
\end{align*} is concave.
\end{proposition}

\noindent {\it Proof.} Let $F:\mathbb R^{kn}\to \mathbb R$ be the function defined by:
\begin{align*}
F(x_{11},\ldots,x_{1n}, \ldots, x_{k1}, \ldots, x_{kn} ) = f\left( \sum_{i=1}^k x_{i1}^2, \ldots, \sum_{i=1}^k x_{in}^2 \right).
\end{align*} We may check that $F$ is convex. Hence, if $(\zeta_{ij})_{i,j=1}^{k,n}$ are independent with $\zeta_{ij}\sim N(0,1)$,
by Ehrhard's inequality we have that the mapping $t\mapsto \Phi^{-1}\circ \mathbb P( F(Z) \ls t)$ is concave, where 
$Z=(\zeta_{11}, \ldots, \zeta_{1n}, \ldots, \zeta_{k1}, \ldots, \zeta_{kn}) \sim N({\bf 0}, I_{kn})$. Finally, the observation that
for $j=1,2, \ldots,n$ the random variables $\sum_{i=1}^k\zeta_{ij}^2$ are independent and $\chi^2 (k)$, yields that
\begin{align*}
\mathbb P(F(Z) \ls t) = \mathbb P( f(W) \ls t ),
\end{align*} for all $t\in \mathbb R$. \prend

\medskip

As an immediate consequence of Proposition \ref{prop:Ehr-chi} we have the following:

\begin{corollary}
Let $f:\mathbb R^n\to \mathbb R$ be 1-unconditional and convex function. Then, the mapping
\begin{align*}
t\mapsto \Phi^{-1} \circ \nu_1^n (f\ls t),
\end{align*} is concave.
\end{corollary}

\section{An application to finite-dimensional normed spaces}

The purpose of this section is to provide an application of the tightness result on the concentration proved in Theorem \ref{thm:main-2}
in the context of norms. Namely, we show that for any given norm on $\mathbb R^n$ there exists a $5$-equivalent norm say, 
which exhibits optimal Gaussian concentration in terms of its Lipschitz constant. In turn this implies an instability result for the dependence
on $\varepsilon$ in the almost isometric version of the randomized Dvoretzky's theorem. 
In order to give the precise statements we have to recall some definitions.

Let $\|\cdot \|$ be an arbitrary norm on $\mathbb R^n$ and let $X=(\mathbb R^n, \|\cdot\|)$.  
We define the global parameters of $X$:
\begin{align*}
b(X) :=\max\{ \|\theta\| : \|\theta\|_2=1 \}, \quad   k(X) := \left( \mathbb E\|Z\| /b(X) \right)^2, \quad Z\sim N({\bf 0},I_n).
\end{align*} The parameter $k(X)$ is usually referred to as the {\it critical dimension} of the normed space $X$. 

First we show the following instability result for the concentration of norms:

\begin{theorem} \label{thm:main-3-1}
There exists an universal constant $C\gr 1$ with the following property: for any $n\gr 1$ and for any norm $\|\cdot\|$ on $\mathbb R^n$, 
there exists a $5$-equivalent norm $\vertiii{ \cdot }$ such that
\begin{align*}
\mathbb P \left( \big| \vertiii{Z}-\mathbb E \vertiii{ Z} \big| > \varepsilon \right) \gr \frac{1}{C} \exp(-C\varepsilon^2/ b(Y)^2 ), \quad Z\sim N({\bf 0},I_n),
\end{align*} for all $\varepsilon>0$, where $Y=(\mathbb R^n, \vertiii{ \cdot })$.
\end{theorem}

At this point we should mention that, although for any given norm there exists a 5-equivalent norm 
for which the classical concentration is optimal, there are several examples
(established recently) which show that choosing appropriately the position of the norm (via a linear map) 
one can exhibit better tail estimates; see \cite{PV-Lp} and \cite{Tik-unc}.

\subsection{Tilted norms and instability of the concentration.} 

Given any norm $\|\cdot\|$ on $\mathbb R^n$, let $x_0^\ast \in S_{X^\ast}$ such that $\|x_0^\ast\|_2=b(X)$, 
where $X=(\mathbb R^n, \|\cdot\|)$. For any $t>0$ we define the norm
\begin{align} \label{eq:norm-tilted}
f_t(x) = \|x\| +t |\langle x, x_0^\ast \rangle|, \quad x\in \mathbb R^n.
\end{align} Let also $X_t = (\mathbb R^n, f_t)$ be the induced normed space. In the next easily verified lemma we collect
some of the basic properties of the norms $f_t$:

\begin{lemma} \label{lem:proper-f-t}
Let $X=(\mathbb R^n, \|\cdot\|)$ be a normed space and let $(f_t)_{t>0}$ be the family of norms defined above. 
\begin{itemize}

\item [a.] For all $x\in \mathbb R^n$ and $t>0$ we have: $\|x\|\ls f_t(x) \ls (1+t) \|x\|$.

\item [b.] For all $t>0$ we have: $b_t\equiv \sup\{ f_t(\theta) : \|\theta\|_2=1 \} =(1+t)b={\rm Lip}(f_t)$ and $k_t\equiv k(X_t) = (1+t)^{-2} (\sqrt{k(X)} + t \sqrt{2/\pi})^2$.

\item [c.]  For $t\gr 4$, we have
$\sqrt{ {\rm Var}[f_t(Z)] } \gr \frac{1}{8} {\rm Lip}(f_t) $.

\end{itemize} 
\end{lemma}

\noindent {\it Proof.} a. Note that for all $x$, $\|x\| \ls f_t(x) \ls \|x\|+ t \|x\| \cdot \|x_0^\ast \|_\ast \ls (1+t) \|x\|$.

\smallskip

\noindent b. Let $\|\theta_0\|_2=1$ with $\langle x_0^\ast,\theta_0\rangle=b$. Then, $f_t(\theta_0)=b+tb=(1+t)b$. In addition, we have:
\begin{align} \label{eq:exp-f-t}
\mathbb E f_t(Z) = \mathbb E\|Z\| +tb \mathbb E|\zeta| = b\sqrt{k} +tb \sqrt{2/\pi}, \quad \zeta\sim N(0,1).
\end{align}
c. We use the inequality $ \sqrt{ {\rm Var}(\xi_1) } -\sqrt{{\rm Var}(\xi_2)} \ls \sqrt{{\rm Var}(\xi_1+\xi_2)}$ to write:
\begin{align*}
\sqrt{ {\rm Var} [f_t(Z)] } &\gr \sqrt{ {\rm Var}(t |\langle Z,x_0^\ast \rangle|) } - \sqrt{ {\rm Var}\|Z\| }  
= t\|x_0^\ast \|_2 \sqrt{ {\rm Var}|\zeta| } - \sqrt{ {\rm Var}\|Z\|} 
\gr t b\sqrt{ {\rm Var}|\zeta| }  - b .
\end{align*} We choose $t\gr 4>2/\sqrt{ {\rm Var}|\zeta| }$ to conclude.

\prend

\medskip

Theorem \ref{thm:main-3-1} immediately follows from the next result:

\begin{proposition} \label{prop:instable-f-t}
Let $\|\cdot\|$ be a norm on $\mathbb R^n$ and let $t\gr 4$. The tilted norms $(f_t)$ defined in \eqref{eq:norm-tilted} satisfy:
\begin{align*}
\mathbb P ( | f_t(Z)-\mathbb E f_t(Z)| \gr \varepsilon ) \gr c_2e^{-C_2 \varepsilon^2/ b_t^2 },
\end{align*} for all $\varepsilon>0$, where $C_2,c_2>0$ are universal constants.
\end{proposition}

\noindent {\it Proof.} From Lemma \ref{lem:proper-f-t}, note that 
$\sqrt{ {\rm Var}[f_t(Z)] } \gr \frac{1}{8} {\rm Lip} (f_t)$ for all $t\gr 4$. Thus, by \eqref{eq:rev-conc-lip} we obtain:
\begin{align*}
\mathbb P \left( \left| f_t(Z) - \mathbb E[f_t(Z)] \right| \gr \varepsilon {\rm Lip}(f_t)  \right) \gr c_2 e^{-C_2 \varepsilon^2},
\end{align*} for all $\varepsilon>0$, as required. \prend

\subsection{Random almost spherical sections of convex bodies} 

The critical dimension $k(X)$ of a normed space was introduced by V. Milman in his work \cite{Mil-dvo} on 
the random version of Dvoretzky's theorem \cite{Dvo}:

\begin{theorem} [Dvoretzky 1961, V. Milman 1971] \label{thm:Dvo} 
For any $\varepsilon\in (0,1)$ there exists $\eta(\varepsilon)>0$ with the following property: 
For any $n\gr 1$ and any norm $\|\cdot\|$ on $\mathbb R^n$, the random $k$-dimensional subspace $F$ 
(with respect to the Haar measure $\nu_{n,k}$ on the Grassmannian $G_{n,k}$) satisfies with high probability
\begin{align}
(1-\varepsilon) M \ls \|y\| \ls (1+\varepsilon) M \|y\|_2 , \quad y\in F,
\end{align} as long as $k\ls \eta(\varepsilon) k(X)$, where $M = \int_{S^{n-1}} \|\theta\|\, d\sigma(\theta)$ and $\sigma$ is the 
uniform probability measure on $S^{n-1}$.
\end{theorem}

V. Milman's proof provides $\eta(\varepsilon) \simeq \varepsilon^2/ \log(1/\varepsilon)$ while 
Gordon \cite{Go} and Schechtman \cite{Sch-eps} proved
that one can always have $\eta(\varepsilon) \simeq \varepsilon^2$. 

For any given normed space $X=(\mathbb R^n, \|\cdot\|)$ and $\varepsilon\in (0,1)$ we define $k(X,\varepsilon)$ the maximal positive 
integer $k\ls n$ for which the random $k$-dimensional subspace $F$ of $X$ is {\it $(1+\varepsilon)$-spherical}, i.e.
\[ \max_{z\in S_F} \|z\| / \min_{z\in S_F} \|z\| <1+\varepsilon, \]
with probability at least 2/3.

With this terminology Theorem \ref{thm:Dvo} implies that $k(X,\varepsilon) \gr c\varepsilon^2 k(X)$, 
for any normed space $X=(\mathbb R^n, \|\cdot\|)$. Note
that there are spaces for which the dependence on $\varepsilon$ in the above asymptotic formula is much better, 
e.g. $k(\ell_\infty^n, \varepsilon)\simeq \frac{\varepsilon}{\log (1 / \varepsilon)} k(\ell_\infty^n)$; see \cite{Sch-cube} and \cite{Tik-cube}.
The reader may consult \cite{MS} for further background on the local theory of normed spaces.

Using the construction introduced in previous paragraph we may show that for any norm $\|\cdot \|$ on $\mathbb R^n$ 
there exists a $t$-equivalent norm $f_t(\cdot)$ such that $k(X_t, \varepsilon)\simeq \varepsilon^2 k(X_t)$ for all, not so large, $t>0$.
In geometric language this can be interpreted as follows: in the space of centrally symmetric $n$-dimensional convex bodies, 
the ones which admit random almost spherical sections, with high probability, in dimension at most
$C\varepsilon^2 k(X)$ form a $C_0$-net with respect to the geometric distance. More precisely we prove the following:

\begin{theorem} [instability]
There exists an universal constant $C>1$ with the following property: For any normed space 
$X=(\mathbb R^n, \|\cdot\|)$ with $C \ls t \ls \sqrt{k(X)}$, the normed spaces $X_t=(\mathbb R^n, f_t)$ 
defined in \eqref{eq:norm-tilted} satisfy:
\begin{itemize}

\item [a.] For all $x\in \mathbb R^n$ we have $\|x\|\ls f_t(x)\ls 2t \|x\|$. 

\item [b.] For every $\varepsilon\in (0,1/3)$ one has $k(X_t, \varepsilon) \simeq  \varepsilon^2k(X_t)$.

\end{itemize}
\end{theorem}

\noindent {\it Proof.} The argument follows the same lines as in \cite{Sch-cube} (see also \cite[Section 5]{PVZ}), hence we 
roughly sketch the details for reader's convenience. 
Let $\varepsilon \in (0,\frac{1}{3})$ and let the set of $k$-dimensional subspaces of $X_t$,
\begin{align*}
{\mathcal F}_\varepsilon:= \left\{ F\in G_{n,k} \mid (1+\varepsilon)^{-1} M_F\ls f_t(\theta) \ls (1+\varepsilon)M_F\, \; \forall \theta\in S_F \right\},
\end{align*} where $M_F= \int_{S_F} f_t(u) \, d\sigma_F(u)$ and $\sigma_F$ denotes the uniform probability measure on the 
sphere $S_F=S^{n-1}\cap F$. Define further,
\begin{align*}
{\mathcal B}_\varepsilon:= \left\{ F\in {\mathcal F}_\varepsilon \mid (1-2\varepsilon)\frac{\mathbb E[f_t(Z)]}{\mathbb E\|Z\|_2}\ls 
M_F\ls (1+2\varepsilon)\frac{\mathbb E[f_t(Z)]}{\mathbb E\|Z\|_2}\right\}.
\end{align*} 
An application of \cite[Lemma 1]{Sch-cube} yields:
\begin{align*}
\nu_{n,k}({\mathcal F}_\varepsilon) &=  \nu_{n,k}({\mathcal F}_\varepsilon \setminus {\mathcal B}_\varepsilon)+ \nu_{n,k}({\mathcal B}_\varepsilon) \\
& \ls \left[ \mathbb P \left( \left\{ f_t(Z) \gr \frac{1+ 2 \varepsilon}{1+\varepsilon}\frac{\mathbb E[f_t(Z)]}{\mathbb E\|Z\|_2} \|Z\|_2  \; {\rm \bf or} \; 
f_t(Z) \ls (1+ \varepsilon)(1- 2\varepsilon)\frac{\mathbb E[f_t(Z)]}{\mathbb E\|Z\|_2} \|Z\|_2 \right\} \right) \right]^k + \\
& \left[ \mathbb P\left( \left\{ \frac{1-2\varepsilon}{1+\varepsilon} \|Z\|_2 \frac{\mathbb E[f_t(Z)]}{\mathbb E\|Z\|_2} \ls f_t(Z) \ls (1+\varepsilon)(1+2\varepsilon)  \frac{\mathbb E [f_t(Z)]}{\mathbb E \|Z\|_2} \|Z\|_2 \right\}\right) \right]^k.
\end{align*} In order to proceed we will need the following estimate:

\begin{lemma} \label{lem:gauss-lem-4}
For any $C_4\ls t\ls \sqrt{k(X)}$ and for every $0<\delta<1/3$ we have:
\begin{align*}
c_4 e^{-C_4\delta^2 k_t } \ls \mathbb P \left( f_t(Z) \ls \frac{(1-\delta)\mathbb E [f_t(Z)]}{\mathbb E\|Z\|_2} \|Z\|_2 \; {\rm \bf or} \;  
f_t(Z) \gr \frac{(1+\delta)\mathbb E [f_t(Z)]}{\mathbb E\|Z\|_2} \|Z\|_2 \right) \ls C_4 e^{-c_4 \delta^2 k_t }.
\end{align*}
\end{lemma}

Taking Lemma \ref{lem:gauss-lem-4} for granted it suffices to consider $1/ \sqrt{k_t}<\varepsilon<1/3$, hence we obtain:
\begin{align*}
\nu_{n,k}({\mathcal F}_\varepsilon) &\ls C_5^k e^{-k \varepsilon^2 k_t/ C_5}+ (1- C_5^{-1} e^{-C_5 \varepsilon^2 k_t} )^k \\
&\ls e^{-\frac{1}{2C_5} k \varepsilon^2 k_t} +1-C_5^{-1} e^{-C_5 \varepsilon^2 k_t },
\end{align*} provided that $\varepsilon^2 k_t > 2 C_5\log C_5$. Now assuming that 
$\nu_{n,k}({\mathcal F}_\varepsilon) \gr 1- e^{-\beta k} \gr 2/3$ 
for some universal constant $\beta>0$ and restricting further $\max\{ 2 \beta C_5, 2C_5\log C_5\} <\varepsilon^2 k_t$, we obtain:
\begin{align*} 
1-e^{-2C_5\varepsilon^2 k_t} \gr 1-C_5^{-1}e^{-C_5 \varepsilon^2 k_t} \gr 1-e^{-\beta k} -e^{-\frac{1}{2C_5} k \varepsilon^2 k_t }  \gr 1- 2e^{- \beta k} 
\gr 1-e^{-c_0 \beta k},
\end{align*} which implies $k\ls \frac{2C_5}{c_0\beta} \varepsilon^2 k_t$, as required. \prend

\medskip

\noindent {\it Proof of Lemma \ref{lem:gauss-lem-4}.} Let $\xi= f_t(Z)/\mathbb E[f_t(Z)]$ and $\eta= \|Z\|_2/\mathbb E\|Z\|_2$.
For any $s\in (0,1)$ we define the sets:
\begin{align*}
A_s=\{|\xi-\eta| > s\eta\}, \quad B_s=\{ |\eta-1|>s \}, \quad \Gamma_s=\{ |\xi-1|>s\}.
\end{align*} Our aim is to show that: 
\begin{align*}
c_4 e^{-C_4\delta^2 k_t } \ls \mathbb P ( A_\delta ) \ls C_4 e^{-c_4 \delta^2 k_t }.
\end{align*} Note that for any $0<s<1/2$ we have
\begin{align*}
\mathbb P(A_s) \ls \mathbb P(\Gamma_{s/4} ) + \mathbb P(B_{s/2}).
\end{align*} Using the Gaussian concentration for $f_t(\cdot)$ and $\|\cdot\|_2$ we infer:
\begin{align*}
\mathbb P(A_\delta) \ls C_0e^{-c_0\delta^2 k_t} +C_0 e^{-c_0\delta^2 n} \ls 2C_0 e^{-c_0 \delta^2 k_t}, \quad 0<\delta < 1/3,
\end{align*} where we have also used the general fact that $k(X)\ls n$. This proves the rightmost inequality.

For the lower estimate note that for any $0<s<1$ one has
\begin{align*}
\mathbb P(A_{s/3}) \gr \mathbb P(\Gamma_s)- \mathbb P(B_{s/3}).
\end{align*} Using the Gaussian concentration for $\|\cdot\|_2$ and Proposition \ref{prop:instable-f-t} we obtain:
\begin{align*}
\mathbb P(A_\delta) \gr c_2e^{-9C_2 \delta^2 k_t} - C_0e^{-c_0\delta^2n} ,
\end{align*} provided that $t\gr 4$, where $C_2>0$ is the constant from Proposition \ref{prop:instable-f-t}. The latter
is larger than $\frac{c_2}{2} e^{-9C_2 \delta^2 k_t}$ provided that $(c_0n-9C_2k_t)\delta^2 \gr \log (2C_0/c_0)$. Note 
that since $4 \ls t \ls \sqrt{k(X)}$, Lemma \ref{lem:proper-f-t}.b yields $k_t\ls C_3 t^{-2}k(X) \ls C_3n/t^2$,
thus it suffices to have $(c_0n-9C_2C_3 t^{-2}n)\delta^2 \gr \log(2C_0/c_0)$. The last one holds 
if $t\gr \max\{4,\sqrt{18C_2C_3/c_0} \}$ and $\delta \gr \sqrt{ \frac{2}{c_0n} \log (\frac{2C_0}{c_0})} $.
The assertion of the lemma follows.  \prend

\bigskip

\noindent {\bf Acknowledgements.} The author would like to thank Grigoris Paouris for posing him the question about the 
tightness of the concentration and for many fruitful discussions. 
He would also like to thank Peter Pivovarov for useful advice and comments, Ramon van Handel and Emanuel Milman 
for valuable remarks.


\bigskip
\bibliography{gauss-tight-ref}
\bibliographystyle{alpha}

\end{document}